\documentclass[12pt]{amsart}
\usepackage[utf8]{inputenc}

\title{A note on geometric theories of fields}
\author{Will Johnson and Jinhe Ye}
\email{willjohnson@fudan.edu.cn, jinhe.ye@maths.ox.ac.uk}
\date{\today}

\usepackage{amsmath, amssymb, amsthm} 
\usepackage{fullpage} 
\usepackage{hyperref} 
\usepackage{centernot} 
\usepackage{xcolor}

\newcommand{\Zz}{\mathbb{Z}}
\newcommand{\Nn}{\mathbb{N}}
\newcommand{\Qq}{\mathbb{Q}}
\newcommand{\Mm}{\mathbb{M}}
\newcommand{\Kk}{\mathbb{K}}

\newcommand{\ba}{{\bar{a}}}
\newcommand{\bb}{{\bar{b}}}
\newcommand{\bc}{{\bar{c}}}

\newcommand{\bx}{{\bar{x}}}
\newcommand{\by}{{\bar{y}}}

\newcommand{\alg}{\mathrm{alg}}

\DeclareMathOperator{\Aut}{Aut}
\DeclareMathOperator{\Gal}{Gal}

\DeclareMathOperator{\tp}{tp}

\DeclareMathOperator{\acl}{acl}
\DeclareMathOperator{\dcl}{dcl}
\newcommand{\eq}{\mathrm{eq}}
\newcommand{\Th}{\mathrm{Th}}

\DeclareMathOperator*{\ind}{\raise0.2ex\hbox{\ooalign{\hidewidth$\vert$\hidewidth\cr\raise-0.9ex\hbox{$\smile$}}}}

\newtheorem{theorem}{Theorem}[section]
\newtheorem{lemma}[theorem]{Lemma} 
\newtheorem{corollary}[theorem]{Corollary} 
\newtheorem{proposition}[theorem]{Proposition}
\newtheorem{fact}[theorem]{Fact}

\theoremstyle{definition} 
\newtheorem{definition}[theorem]{Definition}
\newtheorem{remark}[theorem]{Remark}

\newtheorem{question}[theorem]{Question}

\theoremstyle{remark}
\newtheorem*{acknowledgment}{Acknowledgments}

\begin{document}
	
	\maketitle
	
	\begin{abstract}
		Let $T$ be a complete theory of fields, possibly with
                extra structure.  Suppose that model-theoretic
                algebraic closure agrees with field-theoretic
                algebraic closure, or more generally that
                model-theoretic algebraic closure has the exchange
                property.  Then $T$ has uniform finiteness, or
                equivalently, it eliminates the quantifier
                $\exists^\infty$.  It follows that very slim fields in
                the sense of Junker and Koenigsmann are the same thing
                as geometric fields in the sense of Hrushovski and
                Pillay.  Modulo some fine print, these two concepts
                are also equivalent to algebraically bounded fields in
                the sense of van den Dries.
                
                From the proof, one gets a one-cardinal theorem for geometric theories of fields: any infinite definable set has the same cardinality as the field.  We investigate whether this extends to interpretable sets.  We show that positive dimensional interpretable sets must have the same cardinality as the field, but zero-dimensional interpretable sets can have smaller cardinality.  As an application, we show that any geometric theory of fields has an uncountable model with only countably many finite algebraic extensions.
	\end{abstract}
	
	\section{Introduction}
Throughout the paper, $T$ denotes a complete theory. We use $\acl(-)$ to denote the model-theoretic algebraic closure. When $T$ expands the theory of fields, we use $(-)^\alg$ to denote the field-theoretic algebraic closure. Following Hrushovski, Pillay,
        and Gagelman~\cite{udi-anand-group-field,gagelman}, we say
        that $T$ is \emph{geometric} if (1)--(2) hold and
        \emph{pregeometric} if (1) holds:
        \begin{enumerate}
        \item $\acl(-)$ satisfies the exchange property.
        \item $T$ eliminates $\exists^\infty$, or equivalently, $T$
          has uniform finiteness.
        \end{enumerate}
	A \emph{(pre)geometric structure} is a structure $M$ whose
        complete theory is (pre)geometric.

        Using a simple argument, we show that pregeometric fields are
        geometric (Theorem~\ref{thm1}).  This seems to not be
        well-known.  For example, it is implicitly unknown in
        \cite[Remark~2.12]{udi-anand-group-field}, and a special case
        of this implication is asked as an open problem in
        \cite[Question~1]{JK-slim}.

        As a consequence of Theorem~\ref{thm1}, several concepts in
        the literature are equivalent, namely the \emph{very slim}
        fields of Junker and Koenigsmann
        \cite[Definition~1.1]{JK-slim}, the \emph{geometric fields} of
        Hrushovski and Pillay
        \cite[Remark~2.10]{udi-anand-group-field}, and (modulo some
        fine print) the \emph{algebraically bounded fields} of van den
        Dries~\cite{lou-dimension}.

        Our method also shows that in a geometric field, any infinite
        definable set has the same cardinality as the field
        (Proposition~\ref{card1}), which may be of independent interest.  It is natural to ask whether this extends to interpretable sets.  In Proposition~\ref{card23}, we show that if $X$ is an interpretable set of positive dimension, then $|X| = |K|$, but there are models where all the zero-dimensional interpretable sets satisfy $|X| < |K|$.  As an application, there is an uncountable model $K$ with only countably many finite algebraic extensions (Corollary~\ref{final}), which may be of interest to field theorists.  In the special case of $\omega$-free perfect PAC fields, this recovers examples such as~\cite[Example 2.2]{full-hilb}.

        \begin{acknowledgment}
		The first author was supported by the National Natural Science Foundation
		of China (Grant No.\@ 12101131). The second author was partially supported by GeoMod AAPG2019 (ANR-DFG), \textit{Geometric and
Combinatorial Configurations in Model Theory} and the National Science Foundation under Grant No. DMS-1928930 while the second author participated in a program hosted by the Mathematical Sciences Research Institute in Berkeley, California during Summer 2022. The second author would also like to thank Arno Fehm for interesting discussions.
	\end{acknowledgment}
                
        \section{Uniform finiteness from the exchange property}
        If $M$ is a geometric structure, there is a well-established
        dimension theory on $M$ defined as follows.  If $A\subseteq M$
        and $a$ is a tuple, we define $\dim(a/A)$ to be the length of
        the maximal $\acl_A$-independent subtuple of $a$.  This is
        well-defined by the exchange property.  If $X$ is an
        $A$-definable subset of $M^n$, we define $\dim(X) =
        \max\{\dim(x/A) : x \in X(\mathbb{U})\}$ for some monster model
        $\mathbb{U}\succeq M$.
        \begin{fact}\label{fact:dim}
          Let $M$ be a pregeometric structure, and $X, Y$ be definable sets.
          \begin{enumerate}
          \item $\dim(X)$ is well-defined, independent of the choice
            of $A$.
          \item \label{d2} $\dim(X) > 0 \iff |X| = \infty$.
          \item $\dim(X \times Y) = \dim(X) + \dim(Y)$.
          \item \label{d4} $\dim(M^n) = n$ unless $M$ is finite.
          \item If $f:X \to Y$ is a definable injection or
            surjection, then $\dim(X) \le \dim(Y)$ or $\dim(X) \ge
            \dim(Y)$, respectively.
          \end{enumerate}
        \end{fact}
        Gagelman observed that the dimension theory can also be
        extended to $M^\eq$.  If $a \in M^\eq$, then $\dim(a/A)$ is
        defined to be $\dim(b/Aa) - \dim(b/A)$ for any tuple $b$ with
        $a \in \acl^\eq(Ab)$.  If $X$ is an interpretable set, then
        $\dim(X)$ is defined as for definable sets.  By
        \cite[Lemma 3.3]{gagelman}, these definitions are well-defined, and
        all of Fact~\ref{fact:dim} holds \emph{except for} (\ref{d4})
        and the $\Leftarrow$ direction of (\ref{d2})
        \cite[Proposition 3.4 and p. 321]{gagelman}.

        \begin{lemma} \label{perfect}
          Let $K$ be a pregeometric field.  Then $K$ is perfect.
        \end{lemma}
        \begin{proof}
          Otherwise, $K$ is an infinite field of characteristic $p$.
          Take $a \in K \setminus K^p$.  Then the map
          \begin{equation*}
            f(x,y) \mapsto x^p + ay^p
          \end{equation*}
          is a definable injection $f : K \times K \to K$.  But
          $\dim(K^2) = 2 > \dim(K) = 1$, as $K$ is infinite.
        \end{proof}
        \begin{lemma} \label{trick}
          Let $K$ be a field and $X$ be a subset.  Then one of the
          following holds:
          \begin{enumerate}
          \item \label{t1} The set $S = \{(y - y')/(x' - x) : x,y,x',y' \in X, ~
            x \ne x'\}$ equals $K$.
          \item \label{t2} There is an injection $f : X^2 \to K$ of the form
            $f(x,y) = ax + y$.
          \end{enumerate}
        \end{lemma}
        \begin{proof}
          If $a \in K \setminus S$, then the function $f(x,y) = ax+y$
          is injective on $X^2$.
        \end{proof}
        \begin{lemma} \label{trick2}
          Let $K$ be an infinite pregeometric field, and
          let $X \subseteq K$ be definable.  Then $X$ is infinite if
          and only if $K = \{(y-y')/(x'-x) : x,y,x',y' \in X, ~ x \ne
          x'\}$.
        \end{lemma}
         
        \begin{proof}
          If $X$ is infinite, then $\dim(X) > 0$, so $\dim(X^2) \ge 2
          > \dim(K) = 1$ and there are no definable injections $f : X^2
          \to K$.  Therefore case (\ref{t2}) of Lemma~\ref{trick}
          cannot hold, so (\ref{t1}) holds.  Conversely, if $X$ is
          finite, then case (\ref{t1}) of Lemma~\ref{trick} cannot
          hold, as the set $S$ would be finite.
        \end{proof}
	\begin{theorem} \label{thm1}
		Let $T$ be a complete theory of fields, possibly with
                extra structure.  If $T$ is pregeometric, then $T$ is
                geometric.
	\end{theorem}
	\begin{proof}
		It suffices to eliminate $\exists^\infty$.  If 
                models of $T$ are finite, then $\exists^\infty$ is
                trivially eliminated.  If the models of $T$ are
                infinite, then Lemma~\ref{trick2} gives a first-order criterion for telling whether a definable set $X \subseteq K$ is infinite.
        \end{proof}
        The proof of Theorem~\ref{thm1} is the same argument used in~\cite[Obeservation 3.1]{will-canonical}.
       
	\begin{remark}
		Let $(K,+,\cdot,\ldots)$ be a field, possibly with extra structure.
		If $\acl(-)$ agrees with field-theoretic algebraic closure, then
		$\acl(-)$ satisfies exchange. 
	\end{remark}
        \begin{definition}[{Junker-Koenigsmann~\cite[Definition~1.1]{JK-slim}}]
          A field $K$ is \emph{very slim} if $\acl(-)$ agrees with
          field-theoretic algebraic closure in any elementary
          extension of $K$.
        \end{definition}
	\begin{definition}[{Hrushovski-Pillay~\cite[Remark~2.10]{udi-anand-group-field}}]
		A \emph{strongly geometric field} is a perfect field
                $(K,+,\cdot)$ with a geometric theory that is very
                slim.
        \end{definition}
        Hrushovski and Pillay call these \emph{geometric fields}, but
        we prefer the term \emph{strongly geometric} to distinguish
        strongly geometric fields from the more general case of fields
        that are geometric (as structures).  There are geometric
        fields that are not strongly geometric (Remark~\ref{cxf}).
        \begin{corollary}
          A field $K$ is very slim if and only if it is strongly
          geometric.
        \end{corollary}
        \begin{proof}
          If $K$ is very slim, then $K$ is pregeometric, hence
          geometric by Theorem~\ref{thm1} and perfect by
          Lemma~\ref{perfect} (or \cite[Proposition 4.1]{JK-slim}).
        \end{proof}
	\begin{definition}[van den Dries~\cite{lou-dimension}] \label{ab}
		Let $(K,+,\cdot,\ldots)$ be an expansion of a field, and $F$ be a
		subfield.  Then $K$ is \emph{algebraically bounded over $F$} if for
		any formula $\varphi(\bx,y)$, there are finitely many polynomials
		$P_1,\ldots,P_m \in F[\bx,y]$ such that for any $\ba$, \emph{if}
		$\varphi(\ba,K)$ is finite, then $\varphi(\ba,K)$ is contained in
		the zero set of $P_i(\ba,y)$ for some $i$ such that $P_i(\ba,y)$ does not vanish.  Following the convention in \cite{lou-dimension,JK-slim} we say that $K$
		is \emph{algebraically bounded} if it is algebraically bounded over
		$K$.
	\end{definition}
	
	\begin{lemma} \label{ablem}
		Suppose $K$ is algebraically bounded over some subfield $F$.
		\begin{enumerate}
			\item If $A$ is a subset of $K$, then $\acl(AF)$ is the field-theoretic relative
			algebraic closure $(AF)^\alg \cap K$.
			\item $K$ has uniform finiteness.
			\item If $K^*$ is an elementary extension of $K$, then $K^*$ is
			algebraically bounded over $F$.
		\end{enumerate}
	\end{lemma}
	
	\begin{proof}
		\begin{enumerate}
			\item Clearly $(AF)^\alg \cap K \subseteq \acl(AF)$.  Conversely if
			$b \in \acl(AF)$, then $b$ is in a finite set $\varphi(\ba,K)$ for
			some tuple $\ba$ from $AF$.  By algebraic boundedness,
			$\varphi(\ba,K)$ is contained in a finite zero set of some
			polynomial $P(\ba,y)$, where $P(\bx,y) \in F[\bx,y]$.  Therefore
			$b$ is field-theoretically algebraic over $AF$.
			\item For a fixed formula $\varphi(\bx,y)$, let $P_1,\ldots,P_m \in
			F[\bx,y]$ be polynomials as in Definition~\ref{ab}.  Then the
			cardinality of finite sets of the form $\varphi(\ba,K)$ is bounded
			by the maximum of the degrees of the $P_i$'s.
			\item By (2), the theory of $K$ has uniform finiteness, and so
			$\exists^\infty$ is uniformly eliminated across elementary
			extension of $K$.  It follows that for fixed $\varphi,
			P_1,\ldots,P_m$, the following condition is perserved in
			elementary extensions:
			\begin{quote}
				For any $\ba$, if $\varphi(\ba,y)$ defines a finite set, then
				there is $i \in \{1,\ldots,m\}$ such that $P_i(\ba,y)$ has
				finitely many zeros and $\varphi(\ba,y) \rightarrow P_i(\ba,y) =
				0$.
			\end{quote}
			Therefore algebraic boundedness transfers from $K$ to any
			elementary extension $K^*$. \qedhere
		\end{enumerate}
	\end{proof}
	\begin{lemma} \label{2.12}
		Let $K = (K,+,\cdot,\ldots)$ be a field or an expansion of a field.
		Let $F$ be a subfield.  The following are equivalent:
		\begin{enumerate}
			\item In any elementary extension $K^* \succeq K$, field-theoretic
			algebraic closure over $F$ agrees with model-theoretic algebraic
			closure over $F$: if $A \subseteq K^*$ and $b \in \acl(AF)$, then
			$b \in (AF)^\alg$.
			\item Condition (1) holds and $K$ has uniform finiteness, or equivalently,
			$\exists^\infty$ is eliminated in elementary extensions of $K$.
			\item $K$ is algebraically bounded over $F$.
		\end{enumerate}
	\end{lemma}
	\begin{proof}
		(1)$\implies$(2).  Field-theoretic algebraic closure satisfies the
		exchange property.  Therefore $\acl(-)$ satisfies exchange (after
		naming the elements of $F$ as parameters).  By Theorem~\ref{thm1},
		elementary extensions of $K$ eliminate $\exists^\infty$.
		
		(2)$\implies$(3).  Suppose (2) holds but $K$ fails to be
		algebraically bounded over $F$, witnessed by some formula
		$\varphi(\bx,y)$. Using elimination of $\exists^\infty$, we may assume that there is $n\in \mathbb{N}$ such that $|\varphi(\ba,K)|\leq n$ for every $\ba$. For any finite set of polynomials
		$P_1,\ldots,P_m \in F[\bx,y]$, there is $\ba\in K$ such that
		$\varphi(\ba,K)$ is finite, but is not contained in the zero set of $P_i(\ba,y)$ unless $P_i(\ba,y) \equiv 0$. By compactness, there is an elementary extension $K^* \succeq K$ and a tuple $\ba \in K^*$ such that $\varphi(\ba,K^*)$ is finite, but is not contained in the zero set of $P(\ba,y)$ for \emph{any} $P \in F[\bx,y]$ except those with
		$P(\ba,y) \equiv 0$.  Then $\varphi(\ba,K^*)$ contains a point not
		in $F(\ba)^\alg$, contradicting (2).

		(3)$\implies$(1).  Lemma~\ref{ablem} shows that if $K^*$ is an
		elementary extension of $K$, then
		\begin{itemize}
			\item $K^*$ is algebraically bounded over $F$.
			\item Field-theoretic and model-theoretic algebraic agree over $F$.
		\end{itemize}
		Therefore (1) holds.
	\end{proof}
	Specializing to the case where $K$ is a pure field and $F$ is the prime field, we get the
	following:
        \begin{theorem} \label{thm:equiv}
          Let $K$ be a pure field.  Then $K$ is algebraically bounded over
          the prime field if and only if $K$ is very slim.
        \end{theorem}
        We have thus answered~\cite[Question 1]{JK-slim} positively.
        \begin{remark} If $(K,+,\cdot,\ldots)$ is an \emph{expansion} of a field, and $(K,+,\cdot,\ldots)$ is algebraically bounded over the prime field, then the reduct $(K,+,\cdot)$ is also algebraically bounded over the prime field, and so the underlying field $(K,+,\cdot)$ is very slim.
        \end{remark}
        
	\begin{remark} \label{reducts}
	Reducts of geometric structures are geometric structures.  This is folklore, but we include a proof for completeness.  Let $N$ be a geometric structure and $M$ be a reduct.  Without loss of generality, $N$ and $M$ are highly saturated.  Uniform finiteness transfers from $N$ to $M$ in a trivial way: if $M$ fails uniform finiteness, the same definable family fails uniform finiteness in $N$.  Suppose $\acl^M(-)$ does not satisfy exchange.  Then there are some $a,b \in M$ and $C \subseteq M$ such that $a \notin \acl^M(C)$, $b \notin \acl^M(Ca)$, but $a \in \acl^M(Cb)$.  Let $p(x)$ and $q(x,y)$ be $\tp^M(a/C)$ and $\tp^M(a,b/C)$.  The number of realizations of $p(x)$ is large, so we may find $a' \models p$ with $a' \notin \acl^N(C)$.  Similarly, the number of realizations of $q(a',y)$ is large, so we may find a realization $b' \notin \acl^N(Ca')$.  Then $a'b' \equiv_C ab$ in $M$, so $a' \in \acl^M(Cb') \subseteq \acl^N(Cb')$.  Then $a'$ and $b'$ contradict the exchange property in $N$.
	
	Therefore, any reduct of an algebraically bounded field
                          is geometric, though not necessarily
                          strongly geometric.
	\end{remark}

	\begin{remark} \label{cxf}
		In future work, we will give an example of a pure field
		$(K,+,\cdot)$ of characteristic 0 with a subfield $K_0$ such that
		\begin{enumerate}
			\item Field-theoretic algebraic closure and model-theoretic
			algebraic closure agree over $K_0$, and this remains true in
			elementary extensions.
			\item $\acl(\varnothing)$ contains elements of $K_0$ that are
			transcendental over $\Qq$.  In particular, field-theoretic
			algebraic closure and model-theoretic algebraic closuse do not
			agree over $\Qq$.
		\end{enumerate}
		It follows that this field $K$ is algebraically bounded over $K_0$,
		but not over $\Qq$.  In particular, $K$ is algebraically bounded (over $K$) but not
		not very slim and not a strongly geometric field.\footnote{This contradicts the claim in \cite[p.\@ 482]{JK-slim} that algebraically bounded fields are very slim.} Additionally, the field
		$K$ is geometric (by Remark~\ref{reducts}) but not strongly geometric.

	\end{remark}
	Failure of $\acl(\varnothing)$ to be algebraic over the prime field is
	the \emph{only} way an algebraically bounded field can fail to be very
	slim:
	\begin{proposition}\label{prop:bdd-dcl}
		If $K = (K,+,\cdot,\ldots)$ is algebraically bounded, then $K$ is
		algebraically bounded over the subfield $F = \dcl(\varnothing)$.
	\end{proposition}
	\begin{proof}
	    We use the criterion of Lemma~\ref{2.12}(1) to show that $K$ is algebraically bounded over $F$.
	    Embed $K$ into a monster model $\Kk$.  Suppose $b, \bc \in \Kk$ and $b \in \acl(F\bar{c})$.  We must show $b \in F(\bc)^\alg$.  By Remark~\ref{reducts}, $\Th(K)$ is geometric, because it is geometric after naming parameters from $F$.  Replacing $\bc$ by a basis of $\bc$ in the $\acl$-pregeometry, we may assume that the tuple $\bc$ is field- theoretically algebraically independent over $F$. Now suppose for the sake of contradiction that $b \notin F(\bc)^\alg$.  Then $(\bc,b)$ is also field-theoretically algebraically independent over $F$.
	    
	    As $b \in \acl(F \bc) = \acl(\bc)$, there is a 0-definable set $D \subseteq \Kk^n$ with $(b,\bc) \in D$ and $D_{\bc'}$ finite for each $\bc'$.  By algebraic boundedness over $K$, there are finitely many non-zero polynomials $P_1,\ldots,P_i \in K[x,\by]$ such that $D$ is contained in the union of the zero-sets of the $P_i$.
	    
	    Let $M = \Kk^\alg$ and let $V$ be the Zariski closure of $D$ in $M^{n+1}$.  The polynomials $P_i$ show that $V \subsetneq M^{n+1}$, and so $\dim(V) < n+1$.  By elimination of imaginaries in ACF, there is a finite tuple $e$ in $M$ which codes $V$.  Recall that $\Kk$ is perfect by Lemma~\ref{perfect}.  If $\sigma \in \Aut(M/\Kk)=\Gal(\Kk)$, then $\sigma(D) = D$, $\sigma(V) = V$, and $\sigma(e) = e$. As the tuple $e$ is fixed by $\Gal(M/\Kk)$, it must be in the perfect field $\Kk$.
	    
	    If $\sigma$ is any automorphism of $\Kk$, then $\sigma$ can be extended to an automorphism $\sigma'$ of $M$.  The fact that $D$ is 0-definable implies that $\sigma'(D) = \sigma(D) = D$, which then implies $\sigma'(V) = V$ and $\sigma(e) = \sigma'(e) = e$.  Thus $e$ is $\Aut(\Kk)$-invariant, which implies that $e$ is in $F = \dcl(\varnothing)$.  Therefore, in the structure $M$, the $e$-definable set $V$ is in fact $F$-definable.  However, the tuple $(b,\bc) \in D \subseteq V$ is algebraically independent over $F$, so this implies $\dim(V) = n+1$, contradicting the earlier fact that $\dim(V) < n+1$.
	\end{proof}
 \begin{remark}\label{sizeS}
     Algebraically bounded fields are closely related to fields of size at most S in the sense of~\cite{JK-slim}. For fields of size at most S, \cite[Proposition 3.4]{JK-slim} and Lemma~\ref{2.12} show that they are algebraically bounded over $\dcl(\varnothing)$. On the other hand, by Lemma~\ref{2.12} and~\ref{prop:bdd-dcl}, any algebraically bounded field with $\mathrm{tr.deg}(\dcl(\varnothing))\in \Nn$ is of size at most S. 
\end{remark}
	
	\begin{question}
		Is there a pure field $K$ that is geometric, but not algebraically bounded?
	\end{question}

        \section{Cardinalities} \label{sec3}
        Fix a complete geometric theory $T$ expanding the theory of fields, not necessarily algebraically bounded.
	\begin{proposition} \label{card1}
		If $K \models T$ and $X \subseteq K^n$ is an infinite definable set,
		then $|X| = |K|$.
	\end{proposition}
	\begin{proof}
		Clearly $|X| \le |K|$.  We must show $|X| \ge |K|$.
                Replacing $X$ with a projection onto one of the
                coordinate axes, we may assume $X \subseteq K^1$.  By
                Lemma~\ref{trick2}, $K = \{(y-y')/(x'-x) : x,y,x',y'
                \in X, ~ x \ne x'\}$.  Therefore $|K| \le |X|^4 =
                |X|$.
        \end{proof}

    Proposition~\ref{card1} does not generalize to interpretable sets, as exhibited by the example of the value group $\Zz$ in the geometric field $\Qq_p$.
    In Proposition~\ref{card23} below, we will see that the obstruction is precisely the zero-dimensional interpretable sets.  Before proving this, we need a few general lemmas on geometric structures.
    \begin{lemma} \label{ranks}
		Suppose $M$ is a geometric structure and $X$ is an interpretable set in $M$ with
		$\dim(X) = d > 0$.
		\begin{enumerate}
		    \item There is an interpretable set $Y$ and
		finite-to-one interpretable functions $f : Y \to X$ and $g : Y \to
		M^d$ such that the image $g(Y)$ has dimension $d$.
		\item There is an infinite definable set $D$ with $|X| \ge |D|$.
		\end{enumerate}
	\end{lemma}
	\begin{proof}
	Note that (2) implies (1), by taking $D = g(Y)$.  We prove (1).
		Embed $M$ into a monster model $\Mm \succeq M$.  Take $e \in X$ with
		$\dim(e/M) = d$.  Then $e \in \Mm^\eq$.  Every imaginary is
		definable from a real tuple, so there is a real tuple $\ba \in
		\Mm^m$ with $e \in \acl^\eq(M\ba)$.  Replacing $\ba$ with a subtuple, we may assume that $\ba$ is independent over $M$.  By~\cite[Lemma 3.1]{gagelman},
		$\acl(-)$ continues to satisfy exchange after naming the parameter
		$e$.  Therefore, we can meaningfully talk about real tuples being
		independent over $eM$.  Write $\ba$ as $\bb\bc$ where $\bb$ is a
		maximal subtuple that is independent over $eM$.  Then
		$\bc \in \acl(\bb eM)$.  At the same time, $e \in \acl(\bb \bc M)$,
		so $e$ is interalgebraic with $\bc$ over $\bb M$.
		
		Meanwhile, $\dim(\bb/eM) = \dim(\bb/M) = |\bb|$, because $\bb$ is an independent tuple over $eM$.  Then
		$\bb$ and $e$ are independent from each other over $M$, implying $\dim(e/\bb M) =
		\dim(e/M) = d$ by symmetry.  As $e$ and $\bc$ are interalgebraic over $\bb M$,
		we have $\dim(\bc/ \bb M) = d$.  But $\bb\bc$ is an independent tuple over
		$M$, so $\dim(\bc/ \bb M)$ is the length of $\bc$.  Thus $\bc \in M^d$.
		
		As $e$ and $\bc$ are interalgebraic over $\bb M$, there is a $\bb
		M$-interpretable set $Y_0 \subseteq X \times \Mm^d$ such that
		$(e,\bc) \in Y_0$, and the projections $f_0 : Y_0 \to X$ and $g_0 :
		Y_0 \to \Mm^d$ have finite fibers.  By saturation, there is a uniform bound $N$ on the fiber size.  The image $g_0(Y_0)$ is $\bb M$-definable
		and contains the point $\bc$ with $\dim(\bc/ \bb M) = d$.  Therefore
		$g_0(Y_0)$ has dimension $d$.
		
		Now we have the desired configuration $(Y_0,f_0,g_0)$, but defined
		over the parameter $\bb$ outside $M$.  Because $M \preceq \Mm$ and dimension is
		definable in families \cite[Fact~2.4]{gagelman}, we can replace the
		parameter $\bb$ with something in $M$, getting a $M$-definable
		configuration $(Y,f,g)$ in which the fibers of $f$ and $g$ are still bounded in size by $N$.
	\end{proof}
	\begin{definition}
	  \label{largeness}
	  Let $M$ be a structure.  A \emph{definable notion of largeness}\footnote{This is a purely model-theoretic notion and should not be confused with the notion of large fields.} on $M$ is a partition of the $M$-definable sets into two classes---large and small---such that the following axioms hold:
	  \begin{enumerate}
	      \item Any finite set is small.
	      \item Any definable subset of a small set is small.
	      \item If $Y$ is small and $\{X_a\}_{a \in Y}$ is a definable family of small sets, then the union $\bigcup_{a \in Y} X_a$ is small.
	      \item Smallness is definable in families: if $\{X_a\}_{a \in Y}$ is a definable family, then the set $\{a \in Y : X_a \text{ is small}\}$ is definable.
	  \end{enumerate}
	\end{definition}
	If $N \succeq M$, then any definable notion of largeness on $M$ extends in a canonical way to a definable notion of largeness on $N$ by extending the definition according to (4) above.
	\begin{fact}
	    \label{keisler}
	    Let $M$ be a countable structure in a countable language.  Fix a definable notion of largeness on $M$.  Then there is an elementary extension $N \succeq M$ such that if $X$ is definable in $N$, then $X$ is uncountable if and only if $X$ is large.
	\end{fact}
	Fact~\ref{keisler} is essentially Keisler's completeness theorem for $\mathcal{L}(Q)$ \cite[Section~2]{keisler}, but the translation between these settings is sufficiently confusing that we give the details.
	\begin{proof}[Proof (of Fact~\ref{keisler})]
	    Let $T$ be the elementary diagram of $M$.  Let $\mathcal{L}$ be the language of $T$, and let $\mathcal{L}(Q)$ be the language obtained by adding a new quantifier $(Qx)$.  Let $\psi \mapsto \psi^*$ be the map from $\mathcal{L}(Q)$-formulas to $\mathcal{L}$-formulas interpreting $(Qx)\varphi(x,\by)$ as ``the set of $x$ such that $\varphi(x,\by)$ holds is large.''  More precisely,
	    \begin{itemize}
	        \item $\varphi^* = \varphi$ if $\varphi$ doesn't involve the quantifier $Q$.
	        \item $(\varphi \wedge\psi)^* = \varphi^* \wedge \psi^*$, and similarly for the other logical operators including $\exists$ and $\forall$.
	        \item $((Qx)\varphi(x,\by))^*$ is the formula $\psi(\by)$ such that in models $N \models T$,
	        \[ N \models \psi(\bb) \iff \left(\varphi^*(N,\bb) \text{ is large}\right).\]
	    \end{itemize}
	    Let $T'$ be the set of $\mathcal{L}(Q)$-sentences $\varphi$ such that $T \vdash \varphi^*$.  It is straightforward to verify that $T'$ is closed under the rules of inference on pages 6--7 of \cite{keisler}.  For example, the ``axioms of $\mathcal{L}(Q)$'' \cite[p.\@ 6]{keisler} correspond to the axioms in Definition~\ref{largeness}.  By the completeness theorem for $\mathcal{L}(Q)$ \cite[Section~2]{keisler}, there is an $\mathcal{L}$-structure $N$ which satisfies the sentences $T'$, when $(Qx)$ is interpreted as ``there are uncountably many $x$ such that\ldots''  Ignoring the sentences involving $Q$, we see that $N \models T$, and so $N \succeq M$.  Finally, suppose $X = \varphi(N,\bb)$ is definable in $N$.  Let $\psi(\by)$ be the $\mathcal{L}$-formula such that $\psi(\bb)$ holds iff $\varphi(N,\bb)$ is large.  Then $T'$ contains the sentence
	    \[ (\forall \by)[\psi(\by) \leftrightarrow (Qx)\varphi(x,\by)],\]
	    because its image under $(-)^*$ is the tautology
	    \[ (\forall \by)[\psi(\by) \leftrightarrow \psi(\by)].\]
	    Therefore
	    \[ \varphi(N,\bb) \text{ is uncountable} \iff N \models \psi(\bb) \iff \varphi(N,\bb) \text{ is large}. \qedhere\]
	\end{proof}

    \begin{lemma}\label{keisler2}
	Let $T$ be a complete geometric theory in a countable language.  Then there is a model $N \models T$ such that for any interpretable set $X$, $\dim(X) > 0 \iff |X| > \aleph_0$.
	\end{lemma}
	\begin{proof}
	Take a model $M \models T$.  There is a definable notion of largeness on $M^\eq$ in which $X$ is large iff $\dim(X) > 0$.  The requirements of Definition~\ref{largeness} hold by properties of dimension in $M^\eq$ \cite[Propositions~2.8, 2.9, 2.12]{admissible}.
	Applying Fact~\ref{keisler}, we get an elementary extension $N^\eq \succeq M^\eq$ such that if $X$ is definable in $N^\eq$, then $X$ is uncountable iff $\dim(X) > 0$.
	\end{proof}

	\begin{proposition} \label{card23}
	  Let $T$ be a complete, geometric theory of infinite fields, possibly with extra structure.
	  \begin{enumerate}
	      \item If $K \models T$ and $X$ is an interpretable set of positive dimension, then $|X| = |K|$.
	      \item If the language is countable, there is a model $K \models T$ of cardinality $\aleph_1$, such that every zero-dimensional interpretable set is countable.
	  \end{enumerate}
	\end{proposition}
	\begin{proof}
	    \begin{enumerate}
	        \item Suppose $X$ has positive dimension.  Lemma~\ref{ranks} shows that $|X| \ge |D|$ for some infinite definable set $D$.  Then $|D| \ge |X| \ge |K|$ by Proposition~\ref{card1}.  Finally, $|K| \ge |X|$ is clear.
	        \item Lemma~\ref{keisler2} gives an uncountable model $K$ in which every zero-dimensional interpretable set is countable.  By downward L\"owenheim-Skolem, we can replace $K$ with an elementary substructure of cardinality $\aleph_1$.  \qedhere
	    \end{enumerate}
	\end{proof}

        	\section{Finite extensions} \label{sec2}
     
	Let $T$ be a complete, geometric theory of fields, possibly with extra structure, not necessarily algebraically bounded.

	\begin{proposition}\label{fin-ext}
		If $K \models T$ and $n\in \mathbb{N}_{>0}$, then the (interpretable) set of degree $n$ finite extensions has dimension zero.
	\end{proposition}
	\begin{proof}
		By Lemma~\ref{perfect}, $K$ is perfect. Hence any finite extension of $K$ is a simple extension. Let $X$ be the set of irreducible monic polynomials of
		degree $n$. We can regard $X$ as a definable subset of $K^n$ by
		identifying a polynomial $P(x) = x^n + c_{n-1}x^{n-1} + \cdots +
		c_0$ with the $n$-tuple $(c_0,c_1,\ldots,c_{n-1})$.  Let $Y$ be the
		interpretable set of degree $n$ finite extensions.  Let $f : X \to
		Y$ be the map sending $P(x)$ to the extension $K[x]/(P(x))$.
		
		Note that $\dim(X) \le \dim(K^n) = n$.  By the definition of dimension for interpretable sets, it suffices to show that each fiber of $f$ has dimension at least $n$.
		Fix some $b \in Y$ corresponding to a degree $n$ extension $L/K$.
		We claim $f^{-1}(b)$ has dimension $n$.  By identifying $L$ with
		$K^n$, we can regard $L$ as a definable set with $\dim(L) = n$.
		Because $K$ is perfect, there are only finitely many fields between
		$K$ and $L$.  Let $U$ be the union of the intermediate fields.  Then
		$U$ has lower dimension than $n$, so $\dim(L \setminus U) = n$.  The
		elements of $L \setminus U$ are generators of $L$.  The fiber
		$f^{-1}(b)$ is the set of minimal polynomials of elements of $L
		\setminus U$.  Let $\rho : (L \setminus U) \to f^{-1}(b)$
		be the map sending $a \in L \setminus U$ to its minimal polynomial.
		Then $\rho$ is finite-to-one, so $\dim(f^{-1}(b)) \ge \dim(L
		\setminus U) = n$.
	\end{proof}
 We recover the following corollary, which is presumably well-known (for example, it follows from~\cite[Corollary 3.6]{gagelman} and~\cite[Th\'eor\`eme]{corps-chirurgie}).
	\begin{corollary}
		If $T$ eliminates imaginaries, then models of $T$ have bounded
		Galois group---there are only finitely many extensions of degree $n$
		for each $n$.
	\end{corollary}
	\begin{proof}
		Zero-dimensional definable sets are finite.  If elimination of
		imaginaries holds, then zero-dimensional interpretable sets are
		finite.
	\end{proof}
Combining Propositions~\ref{card23} and \ref{fin-ext}, we have the following.
        \begin{corollary} \label{final}
		If $T$ is a geometric theory of infinite fields in a countable language, then there is an uncountable model $K \models T$ with countably many finite extensions of degree $n$ for each $n$.
	\end{corollary}
For example, since perfect PAC fields are geometric~\cite{CH-bounded}, Corollary~\ref{final} recovers results such as~\cite[Example 2.2]{full-hilb}. The fact that all very slim fields satisfy this property might have some field-theoretic consequences.

	\bibliographystyle{amsalpha}
	\bibliography{ref}

\end{document}